\documentclass[12pt]{article}

\usepackage{amsfonts,mathrsfs}
\usepackage{amssymb,amsmath,amsbsy,amsthm}
\usepackage{dsfont}
\usepackage{blkarray}
\usepackage{tikz}
\usepackage{tkz-euclide}
\usepackage{tikz-cd}
\usepackage{enumerate}
\usetikzlibrary{patterns}
\usepackage{ae}
\usepackage{bm}
\usepackage{graphicx}
\usepackage{float}
\usepackage{cite}
\usepackage[d]{esvect}   
\usepackage{mathtools} 
\usepackage{indentfirst}
\usepackage{lineno}
\usepackage{titlesec}
\usepackage{enumitem}
\usepackage{color}
\usepackage{aliascnt}
\usepackage{varioref}
\usepackage{hyperref}
\hypersetup{colorlinks=true,linkcolor=cyan,anchorcolor=cyan,citecolor=red,CJKbookmarks=True,
	,urlcolor=cyan}
\usepackage{cleveref}
\usetikzlibrary{arrows.meta,calc,positioning}

\definecolor{switchblue}{RGB}{28,78,168}
\definecolor{switchteal}{RGB}{0,116,112}
\definecolor{switchred}{RGB}{174,36,33}
\definecolor{switchgray}{RGB}{105,105,105}

\numberwithin{equation}{section}

\def\min{\mbox{\rm min}}
\def\max{\mbox{\rm max}}

\DeclareFontEncoding{LS1}{}{}
\DeclareFontSubstitution{LS1}{stix}{m}{n}
\DeclareSymbolFont{stixletters}{LS1}{stix}{m}{it}
\DeclareMathAccent{\cev}{\mathord}{stixletters}{"91}
\DeclareMathAccent{\vec}{\mathord}{stixletters}{"92}
\DeclareMathAccent{\vecev}{\mathord}{stixletters}{"95}

\def\({\mbox{\rm (}}\def\){\mbox{\rm )}}

\makeatletter

\newcommand{\Rmnum}[1]{\expandafter\@slowromancap\romannumeral #1@}
\makeatother
\newtheorem{theorem}{Theorem}[section]
\newaliascnt{lemma}{theorem}
\newtheorem{lemma}[lemma]{Lemma}
\aliascntresetthe{lemma}

\newaliascnt{proposition}{theorem}
\newtheorem{proposition}[proposition]{Proposition}
\aliascntresetthe{proposition}

\newaliascnt{fact}{theorem}

\aliascntresetthe{fact}

\newaliascnt{definition}{theorem}

\aliascntresetthe{definition}

\newaliascnt{conjecture}{theorem}

\aliascntresetthe{conjecture}

\newaliascnt{corollary}{theorem}
\newtheorem{corollary}[corollary]{Corollary}
\aliascntresetthe{corollary}

\newaliascnt{claim}{theorem}

\aliascntresetthe{claim}

\newaliascnt{problem}{theorem}

\aliascntresetthe{problem}

\newaliascnt{question}{theorem}

\aliascntresetthe{question}

\newaliascnt{remark}{theorem}

\aliascntresetthe{remark}

\newaliascnt{example}{theorem}
\newtheorem{example}[example]{Example}
\aliascntresetthe{example}

\newaliascnt{notation}{theorem}

\aliascntresetthe{notation}

\begin{document}
	
\begin{center}
{\Large\bf Linear relations for Tutte polynomials of semimatroids}\\ [15pt]
	\end{center}
\begin{center}
Houshan Fu\\[8pt]
School of Mathematics and Information Science, Guangzhou University\\
Guangzhou 510006, Guangdong, P. R. China\\[8pt]
E-mail: fuhoushan@gzhu.edu.cn\\[15pt]
\end{center}
	
\begin{abstract}
We establish a family of linear relations for the coefficients of the Tutte polynomial of a semimatroid. We also introduce ranked central sets, a mild generalization of semimatroids, and prove that the same identities hold in this broader setting. The proof relies on the specialization $x=\frac{z}{z-1}$, $y=z$ along the hyperbola $(x-1)(y-1)=1$; under this specialization, the edge subgraph expansion over central sets reduces to the $f$-polynomial of the central set complex. As applications, we obtain explicit formulas for Tutte polynomials of affine hyperplane arrangements and for balanced Tutte polynomials of biased graphs, whose central sets are central subarrangements and balanced edge sets, respectively. Our results recover the generalized Brylawski's identities for Tutte polynomials of ranked sets, matroids and graphs.
\vspace{1ex}
 \vspace{1ex}
\end{abstract}
\section{Introduction}\label{Sec-1}
The study of graph polynomials originated with the chromatic polynomial, which was introduced by Birkhoff \cite{Birkhoff1912} in 1912 for planar graphs in connection with the four-color problem, and later generalized to arbitrary graphs by Whitney via edge subgraph expansions \cite{Whitney1932,Whitney1932-1}. Tutte's two-variable polynomial unified several graph polynomial invariants, including the chromatic polynomial and the flow polynomial, and has since become a central object in the study of graphs, matroids, hyperplane arrangements, and semimatroids. For a finite graph $G=(V,E)$, the {\em Tutte polynomial} admits the following edge subgraph expansion:
\[
T_G(x,y)=\sum_{A\subseteq E}(x-1)^{r(E)-r(A)}(y-1)^{|A|-r(A)},
\]
where $r$ denotes the rank function of $G$. By Tutte's original definition \cite{Tutte1954}, the coefficient of $x^iy^j$ in $T_G(x,y)$ counts the number of spanning forests with internal activity $i$ and external activity $j$, with respect to a fixed total ordering of the edges. This interpretation, now known as the {\em basis activities expansion}, shows that all coefficients of the Tutte polynomial are non-negative integers, and was subsequently generalized to matroids by Crapo \cite{Crapo1969}.

The purpose of this paper is to study generalized Brylawski's identities for Tutte polynomials. Brylawski \cite{BO1992} showed that a family of linear relations holds among the coefficients of the Tutte polynomial not only for graphs, but also for all matroids. More specifically, write the Tutte polynomial $T_M(x,y)$ of a matroid $M=(E,r)$ as $T_M(x,y)=\sum_{i,j\ge 0}t_{ij}x^iy^j$. Brylawski's identity states that, for any integer $0\le k<|E|$,
\[
\sum_{i=0}^{k}\sum_{j=0}^{k-i}\binom{k-i}{j}(-1)^jt_{ij}=0.
\]
In 2015, Gordon \cite{Gordon2015} extended this result by proving that the same relation holds for a much weaker notion of rank function $r:2^E\to\mathbb Z_{\ge0}$, which only requires 
\[
r(A)\leq \min\{r(E),|A|\} \quad \forall A\subseteq E,
\]
without the submodularity and monotonicity assumptions that hold for matroids. Moreover, Gordon also obtained a new identity for the case $k=|E|$:
\[
\sum_{i=0}^{k}\sum_{j=0}^{k-i}\binom{k-i}{j}(-1)^jt_{ij}=(-1)^{|E|-r(E)}.
\]
Most recently, Beke, Cs\'aji, Csikv\'ari and Pituk \cite{BKCP2023} gave a short proof of this result and extended the work of Brylawski and Gordon to all integers $k\ge0$:
\begin{equation}\label{GBI}
\sum_{i=0}^{k}\sum_{j=0}^{k-i}\binom{k-i}{j}(-1)^jt_{ij}=(-1)^{|E|-r(E)}\binom{k-r(E)}{k-|E|},
\end{equation}
where the binomial coefficient $\binom{k-r(E)}{k-|E|}$ is interpreted as $0$ when $k<|E|$. These relations are called the {\em generalized Brylawski identities}. The key observation in their proof is that one can simplify $T_M(x,y)$ by substituting $x=\frac{z}{z-1}$ and $y=z$ into the hyperbola relation $(x-1)(y-1)=1$. Consequently, 
\[
(z-1)^{r(E)}T_M\Big(\frac{z}{z-1},z\Big)=z^{|E|}.
\]
Comparing coefficients of $z^k$ and taking appropriate linear combinations of the resulting equations yields the generalized Brylawski's identities. This idea is also a key ingredient in obtaining our desired results.

Motivated by the earlier work, our main goal is to generalize the identities of Brylawski, Gordon and Beke et al. to Tutte polynomials of semimatroids. Ardila \cite{Ardila2007} extended the Tutte polynomial to semimatroids, which are abstract structures capturing the dependence properties of affine hyperplane arrangements and generalizing matroids. Semimatroids were introduced by Wachs and Walker \cite{Wachs-Walker1986} in terms of the ``geometric lattice'' of closed sets, later developed in terms of subsets of a finite set by Ardila \cite{Ardila2007}, and also discovered independently by Kawahara \cite{Kawahara2004}. 

Let $E$ be a finite set. A nonempty collection $\mathcal{C}$ of subsets of $E$ is called a \emph{simplicial complex} if whenever $A \in \mathcal{C}$, every subset of $A$ also belongs to $\mathcal{C}$. A {\em semimatroid} is a triple $(E,\mathcal{C},r_{\mathcal{C}})$ consisting of a finite set $E$, a nonempty simplicial complex $\mathcal{C}$ on $E$, and a function $r_\mathcal{C}:\mathcal{C}\to\mathbb{Z}_{\ge 0}$, satisfying the following five properties:
\begin{itemize}[leftmargin=1.4cm]
\item[{\rm (R1)}] If $A\in\mathcal{C}$, then $r_{\mathcal{C}}(A)\le|A|$.
\item[{\rm (R2)}] If $A,B\in\mathcal{C}$ and $A\subseteq B$, then $r_{\mathcal{C}}(A)\le r_{\mathcal{C}}(B)$.
\item[{\rm (R3)}] If $A,B\in\mathcal{C}$ and $A\cup B\in\mathcal{C}$, then 
\[
r_{\mathcal{C}}(A\cap B)+r_{\mathcal{C}}(A\cup B)\le r_{\mathcal{C}}(A)+r_{\mathcal{C}}(B).
\]
\item[{\rm (SR4)}] If $A,B\in\mathcal{C}$ and $r_{\mathcal{C}}(A)=r_{\mathcal{C}}(A\cap B)$, then $A\cup B\in\mathcal{C}$.
\item[{\rm (SR5)}] If $A,B\in\mathcal{C}$ and $r_{\mathcal{C}}(A)<r_{\mathcal{C}}(B)$, then $A\cup e\in\mathcal{C}$ for some $e\in B-A$.
\end{itemize}
Every member in $\mathcal{C}$ is referred to as a {\em central set} of the semimatroid $(E,\mathcal{C},r_{\mathcal{C}})$. $E$, $\mathcal{C}$ and $r_{\mathcal{C}}$ are called the {\em ground set}, the {\em collection of central sets} and the {\em rank function} of $(E,\mathcal{C},r_{\mathcal{C}})$, respectively. In addition, all the maximal sets in $\mathcal{C}$ have the same rank, which is denoted by $r(\mathcal{C})$ and called the {\em rank}  of $(E,\mathcal{C},r_{\mathcal{C}})$. In particular, when $\mathcal{C}=2^E$, the axioms (SR4) and (SR5) hold trivially, and the remaining three axioms (R1)--(R3) coincide with the rank function axioms of matroids. Hence, $(E,\mathcal{C},r_{\mathcal{C}})$ reduces to a matroid in this case. For more information about matroids, we refer the reader to Oxley's book \cite{Oxley2011}.

The {\em Tutte polynomial} of a semimatroid $(E,\mathcal{C},r_{\mathcal{C}})$, introduced by Ardila \cite[Definition 8.1]{Ardila2007}, is defined as
\[
T_{\mathcal{C}}(x,y):=\sum_{A\in\mathcal{C}}(x-1)^{r(\mathcal{C})-r_{\mathcal{C}}(A)}(y-1)^{|A|-r_{\mathcal{C}}(A)}.
\]
Ardila further showed that $T_{\mathcal{C}}(x,y)$ satisfies a deletion-contraction relation \cite[Proposition 8.2]{Ardila2007} and admits a basis activities expansion \cite[Theorem 9.5]{Ardila2007}. In particular, $T_{\mathcal{C}}(x,y)$ can be written in the form $T_{\mathcal{C}}(x,y) = \sum_{i\ge 0,\,j\ge 0} t_{ij} x^i y^j$, where each coefficient $t_{ij}$ is a non-negative integer counting bases with internal activity $i$ and external activity $j$. We now present our linear relations for Tutte polynomials of semimatroids.
\begin{theorem}[Generalized Brylawski Identities for Semimatroids]\label{Result1}
Let $(E,\mathcal{C},r_{\mathcal{C}})$ be a semimatroid, and $T_{\mathcal{C}}(x,y) = \sum_{i,j\ge 0} t_{ij} x^i y^j$. Then, for all integers $k\ge0$,
\[
\sum_{i=0}^{k}\sum_{j=0}^{k-i}\binom{k-i}{j}(-1)^jt_{ij}=\sum_{n=0}^{|E|}(-1)^{r(\mathcal{C})+n}f_n\binom{k+n-r(\mathcal{C})}{k},
\]
where $f_n:=\#\{X\in\mathcal{C}:|X|=n\}$. 
\end{theorem}
When $\mathcal{C} = 2^E$, the triple $(E,\mathcal{C},r_\mathcal{C})$ is a matroid. In this case, we have $f_n = \binom{|E|}{n}$ and $r(\mathcal{C}) = r(E)$, and hence
\begin{equation}\label{A=B}
\sum_{n=0}^{|E|}(-1)^{r(\mathcal{C})+n}f_n\binom{k+n-r(\mathcal{C})}{k}=(-1)^{|E|-r(E)}\binom{k-r(E)}{k-|E|}.
\end{equation}
Indeed,
\begin{align*}
\sum_{n=0}^{|E|}(-1)^{r(E)+n}\binom{|E|}{n}\binom{k+n-r(E)}{k}
&=
(-1)^{r(E)}[z^k](1+z)^{k-r(E)}
\sum_{n=0}^{|E|}\binom{|E|}{n}(-(1+z))^n  \\
&=
(-1)^{r(E)}[z^k](1+z)^{k-r(E)}(-z)^{|E|}  \\
&=
(-1)^{|E|-r(E)}\binom{k-r(E)}{k-|E|}.
\end{align*}
Thus, the generalized Brylawski identities for Tutte polynomials of matroids, given in \eqref{GBI}, can be recovered as a special case of \Cref{Result1}.

A further purpose of this paper is to generalize these linear relations to ranked central sets, following the spirit of Gordon's approach to isolate the key properties of semimatroids. This gives rise to a natural rank function generalization of semimatroids, and extends Gordon's original notion of ranked sets by allowing any simplicial subcomplex of the power set, rather than requiring the full power set itself.

A {\em ranked central set} is a triple $(E,\mathcal{C},r_{\mathcal{C}})$ consisting of a finite set $E$, a nonempty simplicial complex $\mathcal{C}$ on $E$, and a function $r_\mathcal{C}:\mathcal{C}\to\mathbb{Z}_{\ge 0}$, satisfying the following properties:
\begin{itemize}[leftmargin=1.4cm]
\item[{\rm (R0)}] $r_\mathcal{C}(\emptyset)=0$.
\item[{\rm (R1)}] $r_{\mathcal{C}}(A)\le|A|$ for any $A\in\mathcal{C}$.
\end{itemize}
The members in $\mathcal{C}$ are called the {\rm central sets} of the ranked central set $(E,\mathcal{C},r_{\mathcal{C}})$,  the function $r_\mathcal{C}$ is called its {\em rank function}, and
\[
r(\mathcal{C}):=\max\{r_\mathcal{C}(A):A\in\mathcal{C}\}
\]
is called its {\em rank}. 

Moreover, we define the {\em generalized Tutte polynomial} $T_\mathcal{C}(x,y)$ of the ranked central set  $(E,\mathcal{C},r_{\mathcal{C}})$ using the rank function, as follows:
\[
T_{\mathcal{C}}(x,y):=\sum_{A\in\mathcal{C}} (x-1)^{r(\mathcal{C})-r_{\mathcal{C}}(A)} (y-1)^{|A|-r_{\mathcal{C}}(A)}.
\]
The assumption (R1) and the definition of $r(\mathcal{C})$ guarantee that $T_{\mathcal{C}}(x,y)$ is a two-variable polynomial. Thus, we can write $T_{\mathcal{C}}(x,y)$ in the form $T_\mathcal{C}(x,y)=\sum_{i,j\ge 0}t_{ij}x^iy^j$. Unlike the usual Tutte polynomial, its coefficients $t_{ij}$ need not be non-negative. For instance, let $E=\{a,b\}$, $\mathcal{C}=\{\emptyset,\{a\},\{b\}\}$, and define $r_\mathcal{C}(\emptyset)=r_\mathcal{C}(\{a\})=0$ and $r_\mathcal{C}(\{b\})=1$. Then $(E,\mathcal{C},r_\mathcal{C})$ is a ranked central set of rank $r(\mathcal{C})=1$, and its generalized Tutte polynomial is
\[
T_\mathcal{C}(x,y) = (x-1)+(x-1)(y-1) +1=xy-y+1.
\]
Thus $t_{01}=-1$. The following theorem extends the generalized Brylawski identities given in \Cref{Result1} to the generalized Tutte polynomials of ranked central sets.

\begin{theorem}[Generalized Brylawski Identities for Ranked Central Sets]\label{Result2}
Let $(E,\mathcal{C},r_{\mathcal{C}})$ be a ranked central set, and $T_{\mathcal{C}}(x,y) = \sum_{i\ge 0,\,j\ge 0} t_{ij} x^i y^j$. Then, for all integers $k\ge0$,
\[
\sum_{i=0}^{k}\sum_{j=0}^{k-i}\binom{k-i}{j}(-1)^jt_{ij}=\sum_{n=0}^{|E|}(-1)^{r(\mathcal{C})+n}f_n\binom{k+n-r(\mathcal{C})}{k},
\]
where $f_n:=\#\{X\in\mathcal{C}:|X|=n\}$. 
\end{theorem}

It is worth noting that when $\mathcal{C} = 2^E$ and $r(\mathcal{C})=r(E)$, the triple $(E,\mathcal{C},r_\mathcal{C})$ is a ranked set in the sense of Gordon, and the equality in \eqref{A=B} still holds. Therefore, the generalized Brylawski identities for Tutte polynomials of ranked sets, given in \eqref{GBI}, can be recovered as a special case of \Cref{Result2}. In addition, a semimatroid is also a ranked central set, and hence we can directly obtain \Cref{Result1} from \Cref{Result2}.

The paper is organized as follows. In \autoref{Sec2}, we prove \Cref{Result2}, deduce \Cref{Result1}, and examine some basic properties of the coefficients $t_{ij}$. In \autoref{Sec3}, we apply \Cref{Result1} to affine hyperplane arrangements and biased graphs.
\section{Proof of  Theorem \ref{Result2} and further consequences}\label{Sec2}
In this section, we prove \Cref{Result2}, deduce \Cref{Result1}, and then examine some basic questions associated with the coefficients $t_{ij}$.
\subsection{Proofs of  Theorems \ref{Result1} and \ref{Result2}}\label{Sec2-1}
In this subsection, we prove \Cref{Result2}. Since every semimatroid is a ranked central set, \Cref{Result1} follows immediately. Let $a,b$ and $m$ be integers. Throughout this section we use the convention
\[
\binom{a}{b}=\frac{a(a-1)\cdots(a-b+1)}{b!}\quad \text{for } b\ge0,
\qquad
\binom{a}{b}=0\quad \text{for } b<0.
\]
We shall also use the Chu--Vandermonde identity \cite{Vandermonde1972,Zhu1933} in the form
\begin{equation}\label{Vandermonde}
\sum_{l\in\mathbb{Z}}\binom{a}{m-l}\binom{b}{l}=\binom{a+b}{m}.
\end{equation}

Let $(E,\mathcal{C},r_{\mathcal{C}})$ be a ranked central set. We consider the polynomial $F_{\mathcal C}(t)$ as the ordinary generating function for the central sets of $\mathcal C$ with respect to cardinality. Equivalently, it is the $f$-polynomial of the simplicial complex $\mathcal C$, where $f_n$ denotes the number of central sets (faces) of size $n$:
\[
F_{\mathcal{C}}(t):=\sum_{A\in\mathcal{C}}t^{|A|}=\sum_{n=0}^{|E|}f_nt^n.
\]
Notice that this polynomial depends only on the underlying simplicial complex  $\mathcal C$, and not on the rank function $r_{\mathcal C}$. Below we simplify the generalized Tutte polynomial $T_\mathcal{C}(x,y)$ using $F_{\mathcal{C}}(t)$ and the method of Beke et al.
\begin{lemma}\label{HyperbolaLemma}
For every ranked central set $(E,\mathcal{C},r_{\mathcal{C}})$, we have
\[
(z-1)^{r(\mathcal{C})}T_{\mathcal{C}}\Big(\frac{z}{z-1},z\Big)=F_{\mathcal{C}}(z-1).
\]
\end{lemma}
\begin{proof}
Let us consider a new variable $z$, and plug in $x =\frac{z}{z-1}$ and $y=z$. Then the contribution of a central set $A\in\mathcal{C}$ to $T_{\mathcal{C}}\left(\frac{z}{z-1},z\right)$ is
\[
(z-1)^{-r(\mathcal{C})+r_{\mathcal{C}}(A)}(z-1)^{|A|-r_{\mathcal{C}}(A)}=(z-1)^{|A|-r(\mathcal{C})}.
\]
Multiplying by $(z-1)^{r(\mathcal{C})}$ and summing over all $A\in\mathcal{C}$ gives \Cref{HyperbolaLemma}.
\end{proof}
We now proceed to verify \Cref{Result2}.
\begin{proof}[Proof of \Cref{Result2}]
Write $r=r(\mathcal{C})$ for brevity. Substituting $x=\frac{z}{z-1}$ and $y=z$ into $T_\mathcal{C}(x,y)=\sum_{i,j\ge 0}t_{ij}x^iy^j$, and applying \Cref{HyperbolaLemma}, we obtain
\begin{equation}\label{KeyA=B}
\sum_{i,j\ge 0}t_{ij}z^{i+j}(z-1)^{r-i}=F_{\mathcal{C}}(z-1).
\end{equation}
Since $0\le r_\mathcal{C}(A)\le\min\{r,|A|\}$ for any $A\in\mathcal{C}$, we have $t_{ij}=0$ whenever $i>r$. Therefore, both sides in \eqref{KeyA=B} are polynomials in $z$. Thus, comparing the coefficients of $z^l$ in both sides of \eqref{KeyA=B} gives
\begin{equation}\label{CoefficientEquation}
\sum_{i,j\ge 0}t_{ij}(-1)^{r-l+j}\binom{r-i}{l-i-j}=[z^l]F_{\mathcal{C}}(z-1).
\end{equation}
Multiplying \eqref{CoefficientEquation} by $(-1)^l\binom{k-r}{k-l}$ and summing over $l=0,\ldots,k$, the left-hand side in \eqref{KeyA=B} becomes
\[
\sum_{i,j\ge 0}t_{ij}(-1)^{r+j}\sum_{l=0}^{k}\binom{k-r}{k-l}\binom{r-i}{l-i-j}.
\]
By \eqref{Vandermonde}, we have
\[
\sum_{l=0}^{k}\binom{k-r}{k-l}\binom{r-i}{l-i-j}=\binom{k-i}{k-i-j}.
\]
This term is zero unless $0\leq i\leq k$ and $0\leq j\leq k-i$; in that range it is equal to $\binom{k-i}{j}$. Hence, the transformed left-hand side in \eqref{KeyA=B} is
\[
(-1)^r\sum_{i=0}^{k}\sum_{j=0}^{k-i}\binom{k-i}{j}(-1)^jt_{ij}.
\]
For the right-hand side in \eqref{KeyA=B}, since $[z^l]F_{\mathcal{C}}(z-1)=\sum_{A\in\mathcal{C}}(-1)^{|A|-l}\binom{|A|}{l}$, we derive
\[
\sum_{l=0}^{k}(-1)^l\binom{k-r}{k-l}\big([z^l]F_{\mathcal{C}}(z-1)\big)=\sum_{A\in\mathcal{C}}(-1)^{|A|}\sum_{l=0}^{k}\binom{k-r}{k-l}\binom{|A|}{l}.
\]
Applying \eqref{Vandermonde} again yields
\[
\sum_{l=0}^{k}\binom{k-r}{k-l}\binom{|A|}{l}=\binom{k+|A|-r}{k}.
\]
Therefore,
\[
\sum_{l=0}^{k}(-1)^l\binom{k-r}{k-l}\big([z^l]F_{\mathcal{C}}(z-1)\big)=\sum_{A\in\mathcal{C}}(-1)^{|A|}\binom{k+|A|-r}{k}.
\]
Multiplying by $(-1)^r$, we get
\[
\sum_{i=0}^{k}\sum_{j=0}^{k-i}
\binom{k-i}{j}(-1)^jt_{ij}
=
\sum_{A\in\mathcal{C}}
(-1)^{r+|A|}
\binom{k+|A|-r}{k}.
\]
Grouping the central sets according to their cardinalities gives
\[
\sum_{i=0}^{k}\sum_{j=0}^{k-i}
\binom{k-i}{j}(-1)^jt_{ij}
=
\sum_{n=0}^{|E|}(-1)^{r+n}f_n\binom{k+n-r}{k},
\]
which completes the proof.
\end{proof}

\subsection{Further consequences}\label{Sec2-2}
In this subsection, we restrict our attention to some basic properties of the coefficients $t_{ij}$. It is natural to ask whether the sum $\sum_{i=0}^{k}\sum_{j=0}^{k-i}\binom{k-i}{j}(-1)^jt_{ij}$ given in \Cref{Result2} is constant for all sufficiently large integers $k$. The following corollary answers this question.
\begin{corollary}\label{Constant}
Let $(E,\mathcal{C},r_{\mathcal{C}})$ be a ranked central set and set  $r=r(\mathcal{C})$. Then, for all $k\ge r$, if $f_n=0$ for all integers $r<n\le|E|$, then the sum $\sum_{i=0}^{k}\sum_{j=0}^{k-i}\binom{k-i}{j}(-1)^jt_{ij}$ is constant and equal to $f_r$.
\end{corollary}
\begin{proof}
We first look at the right-hand side $T_\mathcal{C}(k):=\sum_{n=0}^{|E|} (-1)^{r+n} f_n \binom{k+n-r}{k}$ of \Cref{Result2}. If $k \geq r$ and $n < r$, then $0 \leq k + n - r < k$, and hence $\binom{k+n-r}{k} = 0$. Thus all terms with $n < r$ disappear for $k \geq r$. For $n \geq r$, applying the identity $\binom{k+n-r}{k} = \binom{k+n-r}{n-r}$, we derive that for all $k \geq r$,
\[
T_{\mathcal{C}}(k) = \sum_{n=r}^{|E|} (-1)^{r+n} f_n \binom{k+n-r}{n-r}.
\]
Thus, $T_\mathcal{C}(k)$ is a polynomial in $k$. Let $\rho:=\max\{n:f_n\ne 0\}$. The term with $n=\rho$ is $(-1)^{r+\rho}f_{\rho}\binom{k+\rho-r}{\rho-r}$.
If $\rho>r$, then this term has degree $\rho-r$ and leading coefficient $(-1)^{r+\rho}\frac{f_{\rho}}{(\rho-r)!}\ne 0$. All other terms have smaller degree. Hence the polynomial $T_{\mathcal{C}}(k)$ has degree $\rho-r$. It follows that the polynomial $T_{\mathcal{C}}(k)$ is constant exactly when $\rho=r$, equivalently when $f_n=0$ for all $n>r$. In that case only the term $n=r$ remains, and the constant is $f_r\binom{k}{0}=f_r$. Applying \Cref{Result2}, the corollary holds.
\end{proof}

Let $(E,\mathcal{C},r_{\mathcal{C}})$ be a semimatroid. For any $A\in\mathcal{C}$, its {\em closure} is
\[
{\rm cl}_\mathcal{C}(A)=\big\{e\in E: A\cup\{e\}\in\mathcal{C}, r_\mathcal{C}(A\cup\{e\})=r_\mathcal{C}(A)\big\}.
\]
A central set $A\in\mathcal{C}$ is {\em independent} if $r_{\mathcal{C}}(A)=|A|$, and is called a {\em spanning set} if $r_{\mathcal{C}}(A)=r(\mathcal{C})$. A maximal independent set is referred to as a {\em basis}. Some evaluations of semimatroid Tutte polynomials have combinatorial interpretations as follows:
\begin{itemize}
\item $T_\mathcal{C}(1,1)$ is the number of bases.
\item $T_\mathcal{C}(2,1)$ is the number of independent sets.
\item $T_\mathcal{C}(1,2)$ is the number of spanning sets.
\item $T_\mathcal{C}(2,2)$ is the number of central sets.
\end{itemize}

Note that matroids and semimatroids are examples of ranked central sets. As a byproduct of \Cref{Constant}, we obtain the following result:
\begin{itemize}
\item If every central set in a semimatroid $(E,\mathcal{C},r_{\mathcal{C}})$ is independent, then for all $k\ge r$, the sum $\sum\limits_{i=0}^{k}\sum\limits_{j=0}^{k-i}\binom{k-i}{j}(-1)^jt_{ij}$ is constant and equal to $T_\mathcal{C}(1,1)$.
\item If $M=U_{r,r}$ is the uniform matroid of size $r$ and rank $r$, then for all $k\ge r$, the sum $\sum\limits_{i=0}^{k}\sum\limits_{j=0}^{k-i}\binom{k-i}{j}(-1)^jt_{ij}$ is constant and equal to $T_\mathcal{C}(1,1)$.
\end{itemize}

We conclude this section with some combinatorial interpretations for the coefficient $t_{00}$. The {\em characteristic polynomial} $\chi_\mathcal{C}(t)$ of a semimatroid $(E,\mathcal{C},r_{\mathcal{C}})$ is a special case of the Tutte polynomial $T_\mathcal{C}(x,y)$, defined by
\[
\chi_\mathcal{C}(t):=(-1)^{r(\mathcal{C})}T_\mathcal{C}(1-t,0)=\sum_{A\in\mathcal{C}}(-1)^{|A|}t^{r(\mathcal{C})-r_\mathcal{C}(A)}.
\]
According to \Cref{Result1}, the first identity is
\[
t_{00}=\sum_{n=0}^{|E|}(-1)^{r(\mathcal{C})+n}f_n.
\]
Comparing the expressions of $t_{00}$ and $\chi_\mathcal{C}(t)$, we immediately deduce
\[
t_{00}=(-1)^{r(\mathcal{C})}\chi_\mathcal{C}(1).
\] 
Additionally, the coefficient $t_{00}$ is closely related to an important combinatorial invariant, Crapo's beta invariant. For a matroid $M=(E,r)$, the {\em Crapo's beta invariant} $\beta(M)$ was introduced by Crapo \cite{Crapo1967} and is given by
\[
\beta(M):=(-1)^{r(E)}\sum_{A\subseteq E}(-1)^{|A|}r(A).
\]
It can also be computed from the {\em characteristic polynomial} $\chi_M(t)$ of $M$, defined as
\[
\chi_M(t):=\sum_{A\subseteq E}(-1)^{|A|}t^{r(E)-r(A)},
\]
and the relation between them is 
\[
\beta(M)=(-1)^{r(E)+1}\left.\frac{d \chi_M(t)}{dt}\right|_{t=1}=(-1)^{r(E)+1}\chi'_M(1).
\]
For further details on Crapo's beta invariant, we refer the reader to \cite{Oxley1982}. 

In the study of semimatroids, there are many equivalent ways to define a semimatroid. One can define a semimatroid in terms of a matroid $\tilde{N}$ on the ground set $E\sqcup\{p\}$ with a distinguished element $p$. The pair $(\tilde{N},p)$ is called a {\em pointed matroid}. Ardila \cite[Theorem 5.4]{Ardila2007} showed that an arbitrary pointed matroid $(\tilde{N},p)$ on $E\sqcup\{p\}$ defines a unique semimatroid $(E,\mathcal{C},r_{\mathcal{C}})$ such that 
\[
\mathcal{C}=\big\{A\subseteq E\mid p\notin \operatorname{cl}_{\tilde{N}}(A)\big\}
\]
and $r_{\mathcal{C}}$ is the restriction of $r_{\tilde{N}}$ to $\mathcal{C}$. This correspondence was implicit in the work of Wachs and Walker \cite{Wachs-Walker1986}. Based on this result, Ardila \cite[Proposition 8.7]{Ardila2007} further obtained a close relation between the characteristic polynomials of the semimatroid and the pointed matroid:
\[
\chi_{\tilde{N}}(t)=(t-1)\chi_\mathcal{C}(t).
\]
Furthermore, a routine calculation gives
\[
t_{00}=\beta(\tilde{N}).
\]
Summarizing the above arguments, we obtain the following theorem.
\begin{theorem}\label{Beta-Semi}
Let  $(\tilde N,p)$ be a pointed matroid, and $(E,\mathcal{C},r_\mathcal{C})$ be its corresponding semimatroid. Then
\[
t_{00}=(-1)^{r(\mathcal{C})}\chi_\mathcal{C}(1)=\beta(\tilde{N}).
\]
\end{theorem}
\section{Applications to hyperplane arrangements and biased graphs}\label{Sec3}
In this section, we apply our results to two concrete semimatroid models: affine hyperplane arrangements and biased graphs. We also record the resulting interpretations of the coefficient $t_{00}$.
\subsection{Application to affine hyperplane arrangements}\label{Sec3-1}
An {\em affine hyperplane arrangement} $\mathcal{A}$ is a finite set of affine hyperplanes in a $d$-dimensional vector space $V$ over a field $\mathbb{F}$. A subarrangement $\mathcal{B}\subseteq\mathcal{A}$ is called {\em central} if $\bigcap_{H\in\mathcal{B}}H\ne\emptyset$. Let $\mathcal{C}_{\mathcal{A}}$ denote the collection of central subarrangements of $\mathcal{A}$. For any $\mathcal{B}\in \mathcal{C}_{\mathcal{A}}$, define 
\[
r_{\mathcal{A}}(\mathcal{B}):={\rm codim}\bigcap_{H\in\mathcal{B}}H=d-\dim \bigcap_{H\in\mathcal{B}}H.
\]
As noted in \cite[Proposition 2.2]{Ardila2007}, the triple $(\mathcal{A},\mathcal{C}_{\mathcal{A}},r_{\mathcal{A}})$ forms a semimatroid, where $\mathcal{C}_{\mathcal{A}}$ is the corresponding collection of central sets and $r_\mathcal{A}$ is the corresponding rank function. Its {\em Tutte polynomial} is given by
\[
T_{\mathcal{A}}(x,y):=\sum_{\mathcal{B}\in\mathcal{C}_{\mathcal{A}}}(x-1)^{r(\mathcal{A})-r_{\mathcal{A}}(\mathcal{B})}(y-1)^{|\mathcal{B}|-r_{\mathcal{A}}(\mathcal{B})},
\]
where $r(\mathcal{A})=r(\mathcal{C}_{\mathcal{A}})$. The {\em characteristic polynomial} $\chi_\mathcal{A}(t)$ of $\mathcal{A}$ is a special case of $T_\mathcal{A}(x,y)$, defined as
\[
\chi_{\mathcal{A}}(t):=\sum_{\mathcal{B}\in\mathcal{C}_{\mathcal{A}}}(-1)^{|\mathcal{B}|}t^{\dim\bigcap_{H\in\mathcal{B}}H}=(-1)^{r(\mathcal{A})}t^{d-r(\mathcal{A})}T_\mathcal{A}(1-t,0).
\]

Specializing \Cref{Result1} to affine hyperplane arrangements gives the following corollary.
\begin{corollary}[Generalized Brylawski Identities for Affine Hyperplane Arrangements]\label{CorArrangement}
Let $\mathcal{A}$ be an affine hyperplane arrangement, and  $T_{\mathcal{A}}(x,y)=\sum_{i,j\geq0}t_{ij}x^iy^j$. Then, for all integers $k\ge0$,
\[
\sum_{i=0}^{k}\sum_{j=0}^{k-i}\binom{k-i}{j}(-1)^jt_{ij}=\sum_{n=0}^{|\mathcal{A}|}(-1)^{r(\mathcal{A})+n}c_n\binom{k+n-r(\mathcal{A})}{k},
\]
where $c_n:=\#\{\mathcal{B}\in\mathcal{C}_{\mathcal{A}}:|\mathcal{B}|=n\}$.
\end{corollary}
Furthermore, when $\mathcal{A}$ is central, every subarrangement is central. Thus $c_n=\binom{|\mathcal{A}|}{n}$, and \Cref{CorArrangement} reduces to the generalized Brylawski identity for the central hyperplane arrangement:
\[
\sum_{i=0}^{k}\sum_{j=0}^{k-i}\binom{k-i}{j}(-1)^jt_{ij}=(-1)^{|\mathcal{A}|-r(\mathcal{A})}\binom{k-r(\mathcal{A})}{k-|\mathcal{A}|}.
\]

We are now ready to present a small example to illustrate \Cref{CorArrangement}.
\begin{example}{\rm
Let $\mathcal{A}$ be the affine hyperplane arrangement in $\mathbb R^2$ consisting of the four lines
\[
H_1: x=0,\quad H_2:y=0,\quad H_3: x+y=0,\quad H_4: x=1,
\]
which is depicted as follows:
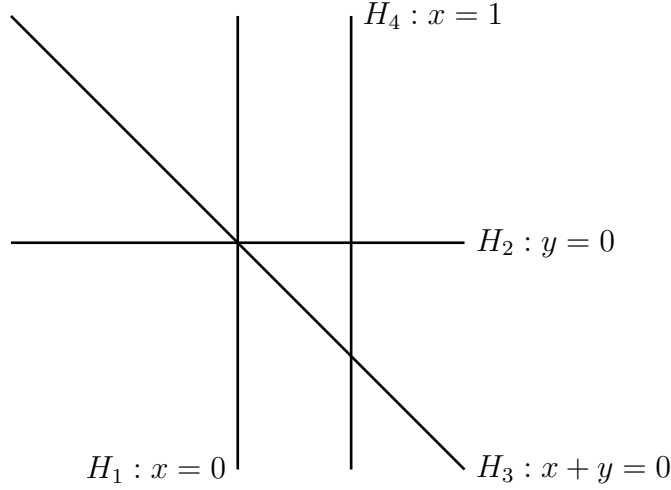
\begin{center}
\begin{tikzpicture}[scale=1.5,line width=1pt]
\draw[black](-2,0)--(2,0) node[right]{$H_2:y=0$};
\draw[black](0,2)--(0,-2) node[left]{$H_1:x=0$};
\draw[black](1,-2)--(1,2) node[right]{$H_4:x=1$};
\draw[black](-2,2)--(2,-2) node[right]{$H_3:x+y=0$};
\end{tikzpicture}
\end{center}
Then $r(\mathcal{A})=2$, and the central subarrangements are the empty subarrangement, the four one-element subarrangements, all two-element subarrangements except $\{H_1,H_4\}$, and the unique three-element subarrangement $\{H_1,H_2,H_3\}$. Hence
\[
c_0=1,\quad c_1=4,\quad c_2=5,\quad c_3=1,\quad c_n=0 \;(n\ge 4).
\]
Then the Tutte polynomial  $T_\mathcal{A}(x,y)$ is
\begin{align*}
T_{\mathcal{A}}(x,y)
&=(x-1)^2+4(x-1)+5+(y-1) \\
&=x^2+2x+y+1.
\end{align*}
Thus, the nonzero coefficients are
\[
t_{20}=1,\qquad t_{10}=2,\qquad t_{01}=1,\qquad t_{00}=1.
\]
\Cref{CorArrangement} gives that
\[
\begin{array}{rcl}
k=0:&
t_{00}=1
&=
1-4+5-1,\\[2mm]
k=1:&
t_{00}-t_{01}+t_{10}=2
&=
\binom{-1}{1}-4\binom{0}{1}+5\binom{1}{1}-\binom{2}{1},\\[2mm]
k=2:&
t_{00}-2t_{01}+t_{10}+t_{20}=2
&=
\binom{0}{2}-4\binom{1}{2}+5\binom{2}{2}-\binom{3}{2}.
\end{array}
\]
More generally, for every $k\geq 2$, both sides are equal to $4-k$.
}
\end{example}

For $k=0$, the first identity takes the form
\[
t_{00}=
\sum_{n=0}^{|\mathcal{A}|}(-1)^{r(\mathcal{A})+n}c_n
\]
Thus, we have
\[
t_{00}=(-1)^{r(\mathcal{A})}\chi_{\mathcal{A}}(1).
\]
When $\mathcal{A}$ is a real affine hyperplane arrangement, Zaslavsky's region-counting formula \cite{Zaslavsky1975} states that $(-1)^{r(\mathcal{A})}\chi_{\mathcal{A}}(1)$ is the number of relatively bounded regions of $\mathcal{A}$. Consequently, $t_{00}$ is the number of relatively bounded regions of $\mathcal{A}$.

As in the semimatroid setting, coning is a basic construction in the study of hyperplane arrangements that relates affine and central arrangements. Suppose an affine hyperplane $H$ in $\mathbb{F}^d$ is defined by $H: a_1x_1+a_2x_2+\cdots+a_dx_d = b$. We then define a linear hyperplane $cH$ in $\mathbb{F}^{d+1}$ by $cH: a_1x_1+a_2x_2+\cdots+a_dx_d - bx_{d+1} = 0$. Let $\mathcal{A}$ be an affine hyperplane arrangement in $\mathbb{F}^d$. The {\em cone} $c\mathcal{A}$ of $\mathcal{A}$ is the central hyperplane arrangement in $\mathbb{F}^{d+1}$ consisting of all $cH$ for $H\in\mathcal{A}$, together with the additional hyperplane $K_0: x_{d+1}=0$. The corresponding characteristic polynomials satisfy
\[
\chi_{c\mathcal{A}}(t)=(t-1)\chi_\mathcal{A}(t),
\]
which can be found in  \cite[Proposition 2.51]{Orlik-Terao1992} and \cite{Stanley2007}. The normal vectors of the linear hyperplanes in $c\mathcal{A}$ naturally determine an $\mathbb{F}$-vector matroid $M_{c\mathcal{A}}$. In this context, the semimatroid determined by the pointed matroid $(M_{c\mathcal{A}}, K_0)$ corresponds precisely to the semimatroid $(\mathcal{A},\mathcal{C}_\mathcal{A},r_\mathcal{A})$ arising from the affine hyperplane arrangement $\mathcal{A}$. Thus, as a direct consequence of  \Cref{Beta-Semi}, we obtain the following corollary.
\begin{corollary}
Let $\mathcal{A}$ be an affine hyperplane arrangement in $\mathbb{F}^d$. Then
\[
t_{00}=(-1)^{r(\mathcal{A})}\chi_{\mathcal{A}}(1)=\beta(M_{c\mathcal{A}}).
\]
When $\mathbb{F}=\mathbb{R}$, $t_{00}$ is the number of relatively bounded regions of $\mathcal{A}$.

\end{corollary}
\subsection{Application to biased graphs}\label{Sec3-2}
In this subsection, we specialize \Cref{Result1} to biased graphs. Biased graphs provide another natural semimatroid model; they were introduced by Zaslavsky \cite{Zaslavsky1987,Zaslavsky1989} as an abstraction of gain graphs that retains their combinatorial structure. Let $G=(V,E)$ be a finite graph. If $G'$ is a subgraph containing no edge $e\in E$, we denote by $G'+e$ the subgraph obtained from $G'$ by adding $e$. For $u,v\in V$, a {\em $uv$-path} is a path from $u$ to $v$. A closed path is called a {\em cycle}, and its edge set is called a {\em circle}. A {\em theta graph} is the union of three internally disjoint open paths with the same two endpoints. A {\em biased graph} $\Omega = (G,\mathcal{B})$ consists of a graph $G$ and a class $\mathcal{B}$ of circles in $G$ satisfying the following {\em theta property}: if two circles of a theta graph belong to $\mathcal{B}$, then so does the third. We call $G$ the {\em underlying graph} of $\Omega$, and $\mathcal{B}$ the set of {\em balanced circles} of $\Omega$. An edge set is {\em balanced} if every circle contained in it is balanced. Moreover, for any $A\subseteq E$, we say that the spanning subgraph $(V,A)$ is {\em balanced} if $A$ itself is balanced. In particular, the cycle corresponding to a balanced circle is called a {\em balanced cycle}. We denote by $\mathcal{C}_\Omega$ the collection of balanced edge sets of $\Omega$. It is clear that if $A\subseteq E$ is balanced, then every subset of $A$ is balanced as well. Thus, $\mathcal{C}_\Omega$ is a simplicial complex. The {\em rank function} $r_\Omega$ of $\Omega$ is given by
\[
r_\Omega(A):=|V|-c(A),\quad\forall\, A\in\mathcal{C}_\Omega,
\]
where $c(A)$ is the number of connected components of the spanning subgraph $(V,A)$. In fact, $r_\Omega$ is the usual rank function $r_G$ of the underlying graph $G$ restricted to balanced edge sets. Thus, the rank $r(\Omega)$ of $\Omega$ is at most the rank $r_G(E)$ of $G$. Since every spanning forest of $G$ is balanced and has rank $|V|-c(E)=r_G(E)$, we obtain $r(\Omega)=|V|-c(E)$.

In 1995, Zaslavsky \cite{Zaslavsky1995} further generalized the classical chromatic and Tutte polynomials of graphs to biased graphs, and proposed the balanced chromatic polynomial and the balanced Tutte polynomial. The {\em balanced chromatic polynomial} $\chi_\Omega^b(t)$  of $\Omega$ is defined by
\[
\chi_\Omega^b(t):=\sum_{A\in\mathcal{C}_\Omega}(-1)^{|A|}t^{r(\Omega)-r_\Omega(A)},
\]
and the {\em balanced Tutte polynomial} $T_\Omega^b(x,y)$ of $\Omega$ is given by
\[
T_\Omega^b(x,y):=\sum_{A\in\mathcal{C}_\Omega}(x-1)^{r(\Omega)-r_\Omega(A)}(y-1)^{|A|-r_\Omega(A)}.
\]
Clearly, $\chi_\Omega^b(t)=(-1)^{r(\Omega)}T_\Omega^b(1-t,0)$, and these polynomials reduce to the corresponding chromatic and Tutte polynomials when $\mathcal{B}$ is the set of all circles of $G$. 

Work of Forge and Zaslavsky \cite{FZ2016} implies that a gain graph yields a semimatroid when its balanced edge sets are taken as central sets, and hence suggests that the triple $(E,\mathcal{C}_{\Omega},r_{\Omega})$ should be a semimatroid. Since we have not found a detailed proof of this claim in the existing literature, we give a self-contained proof below. We first establish the following lemma, which characterizes how balance is preserved under the theta construction.
\begin{lemma}\label{Balanced}
Let $\Omega=(G,\mathcal{B})$ be a biased graph with the underlying graph $G=(V,E)$, $A\in\mathcal{C}_\Omega$, and $e=uv$ be an edge not in $A$. If $(V,A)$ contains a $uv$-path $P$ such that $P+e$ is a balanced cycle, then $A\sqcup\{e\}$ is balanced.
\end{lemma}
\begin{proof}
We first assert that for every $uv$-path $Q$ in $(V,A)$, the cycle $Q+e$ is balanced. For convenience, we denote by $V(G')$ the vertex set of any subgraph $G'$ of $G$, and view every $uv$-path as oriented from $u$ to $v$. If $R$ is an oriented $uv$-path and $x,y\in V(R)$ appear in that order along $R$, we write $R[x,y]$ for the subpath of $R$ from $x$ to $y$. We will transform $P$ into $Q$ by replacing one subpath at a time. More precisely, we will construct a sequence of $uv$-paths $P=P_0,P_1,\ldots,P_m=Q$ in $(V,A)$ with $P_0=P$, such that $P_i+e$ is balanced for every $i$.

\noindent{\bf Step 1: Choose the first place where $Q$ leaves $P_i$.}  Suppose we have constructed $P_i$ in $(V,A)$ with $P_i+e$ already balanced. If $P_i = Q$, we stop. Otherwise, let $a$ be the terminal vertex of the maximal common initial subpath of $P_i$ and $Q$, where the trivial initial subpath consisting of only $u$ is allowed. Thus $a$ always exists; since $P_i \ne Q$, $a$ is not the terminal vertex $v$, and $Q$ leaves $P_i$ immediately after $a$. The two paths may meet again after $a$, and the next step will select the first such meeting.

\noindent{\bf Step 2: Follow $Q$ until it returns to $P_i$.} Starting at $a$ and walking along $Q$, let $b$ be the first vertex after $a$ that lies on $P_i$. This vertex exists because $v\in V(P_i)\cap V(Q)$. Since $Q$ is a  path, it cannot return after $a$ to any vertex of the common initial subpath before $a$. Hence $b$ lies on the subpath $P_i[a,v]$. Consider
\[
S:=Q[a,b],\quad T:=P_i[a,b].
\]
From the choice of $b$, we have
\[
(V(S)\setminus\{a,b\})\cap V(P_i)=\varnothing.
\]
This implies that $S$ and $T$ are internally vertex-disjoint $ab$-paths, and hence $S\cup T$ is a cycle contained in $(V,A)$, as depicted in the following figure. As $A$ is balanced, the cycle $S\cup T$ is balanced.
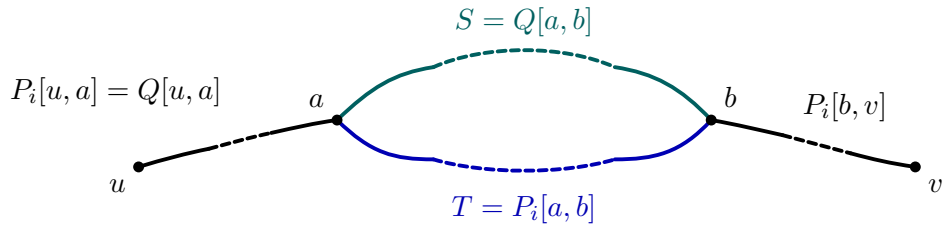
\begin{center}
\begin{tikzpicture}[
  x=1.5cm,y=1cm,
  line cap=round,
  line join=round,
  vertex/.style={circle,fill=black,inner sep=1.45pt},
  solidblack/.style={line width=1.45pt,black},
  dashblack/.style={line width=1.45pt,densely dashed},
  solidblue/.style={line width=1.45pt,blue!70!black},
  dashblue/.style={line width=1.45pt,blue!70!black,densely dashed},
  solidteal/.style={line width=1.45pt,teal!75!black},
  dashteal/.style={line width=1.45pt,teal!75!black,densely dashed},
  solidgray/.style={line width=1.45pt,black},
  dashgray/.style={line width=1.45pt,black,densely dashed},
  lab/.style={font=\small}
]

\coordinate (u) at (0,0);
\coordinate (a) at (1.75,0.62);
\coordinate (b) at (5.05,0.62);
\coordinate (v) at (6.85,0);

\coordinate (l1) at (0.62,0.24);
\coordinate (l2) at (1.17,0.45);

\coordinate (s1) at (2.60,1.32);
\coordinate (s2) at (4.20,1.32);

\coordinate (t1) at (2.60,0.10);
\coordinate (t2) at (4.20,0.10);

\coordinate (r1) at (5.68,0.42);
\coordinate (r2) at (6.28,0.18);

\draw[solidblack] (u) .. controls (0.22,0.10) and (0.42,0.18) .. (l1);
\draw[dashblack]  (l1) .. controls (0.82,0.32) and (0.98,0.39) .. (l2);
\draw[solidblack] (l2) .. controls (1.38,0.52) and (1.55,0.58) .. (a);

\draw[solidteal] (a) .. controls (2.05,1.05) and (2.25,1.24) .. (s1);
\draw[dashteal]  (s1) .. controls (3.05,1.62) and (3.75,1.62) .. (s2);
\draw[solidteal] (s2) .. controls (4.55,1.24) and (4.78,1.05) .. (b);

\draw[solidblue] (a) .. controls (2.05,0.18) and (2.25,0.10) .. (t1);
\draw[dashblue]  (t1) .. controls (3.05,-0.10) and (3.75,-0.10) .. (t2);
\draw[solidblue] (t2) .. controls (4.55,0.10) and (4.78,0.18) .. (b);

\draw[solidgray] (b) .. controls (5.28,0.56) and (5.47,0.49) .. (r1);
\draw[dashgray]  (r1) .. controls (5.88,0.34) and (6.08,0.26) .. (r2);
\draw[solidgray] (r2) .. controls (6.48,0.10) and (6.62,0.05) .. (v);

\node[vertex] at (u) {};
\node[vertex] at (a) {};
\node[vertex] at (b) {};
\node[vertex] at (v) {};

\node[lab,below left=1pt] at (u) {$u$};
\node[lab,above left=1pt] at (a) {$a$};
\node[lab,above right=1pt] at (b) {$b$};
\node[lab,below right=1pt] at (v) {$v$};

\node[lab] at (-.2,1) {$P_i[u,a]=Q[u,a]$};
\node[lab,teal!75!black] at (3.40,1.92) {$S=Q[a,b]$};
\node[lab,blue!70!black] at (3.40,-0.55) {$T=P_i[a,b]$};
\node[lab,black] at (6.23,0.80) {$P_i[b,v]$};

\end{tikzpicture}
\end{center}

\noindent{\bf Step 3: Replace $T$ with $S$.} Let \[P_{i+1}:=P_i[u,a]\cup S\cup P_i[b,v]\] be the new $uv$-path. By construction, $P_{i+1}$ is indeed a $uv$-path in $(V,A)$. Moreover, $P_{i+1}$ agrees with $Q$ from $u$ at least as far as $b$, while $P_i$ only agreed with $Q$ up to $a$. Thus, the length of the common initial subpath of $P_{i+1}$ with $Q$ is strictly larger than that of $P_i$. It remains to verify that $P_{i+1}+e$ is balanced. For this purpose, we consider the following subgraph of $(V,A\sqcup\{e\})$:
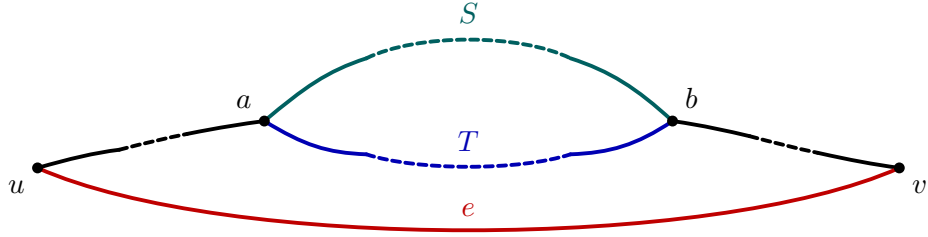
\begin{center}
\begin{tikzpicture}[
  x=1.5cm,y=1cm,
  line cap=round,
  line join=round,
  vertex/.style={circle,fill=black,inner sep=1.45pt},
  solidblack/.style={line width=1.45pt,black},
  dashblack/.style={line width=1.45pt,densely dashed},
  solidblue/.style={line width=1.45pt,blue!70!black},
  dashblue/.style={line width=1.45pt,blue!70!black,densely dashed},
  solidteal/.style={line width=1.45pt,teal!75!black},
  dashteal/.style={line width=1.45pt,teal!75!black,densely dashed},
  edgee/.style={line width=1.45pt,red!75!black},
  lab/.style={font=\small}
]

\coordinate (u) at (0,0);
\coordinate (a) at (2.00,0.62);
\coordinate (b) at (5.60,0.62);
\coordinate (v) at (7.60,0);

\coordinate (l1) at (0.72,0.24);
\coordinate (l2) at (1.32,0.45);

\coordinate (s1) at (2.90,1.45);
\coordinate (s2) at (4.70,1.45);

\coordinate (t1) at (2.90,0.18);
\coordinate (t2) at (4.70,0.18);

\coordinate (r1) at (6.28,0.42);
\coordinate (r2) at (6.88,0.20);

\draw[solidblack] (u) .. controls (0.28,0.12) and (0.50,0.20) .. (l1);
\draw[dashblack]  (l1) .. controls (0.95,0.33) and (1.12,0.39) .. (l2);
\draw[solidblack] (l2) .. controls (1.62,0.54) and (1.80,0.58) .. (a);

\draw[solidteal] (a) .. controls (2.35,1.06) and (2.55,1.28) .. (s1);
\draw[dashteal]  (s1) .. controls (3.35,1.78) and (4.25,1.78) .. (s2);
\draw[solidteal] (s2) .. controls (5.05,1.28) and (5.28,1.06) .. (b);

\draw[solidblue] (a) .. controls (2.35,0.30) and (2.55,0.20) .. (t1);
\draw[dashblue]  (t1) .. controls (3.35,-0.04) and (4.25,-0.04) .. (t2);
\draw[solidblue] (t2) .. controls (5.05,0.20) and (5.28,0.30) .. (b);

\draw[solidblack] (b) .. controls (5.88,0.56) and (6.08,0.49) .. (r1);
\draw[dashblack]  (r1) .. controls (6.48,0.34) and (6.66,0.27) .. (r2);
\draw[solidblack] (r2) .. controls (7.18,0.10) and (7.36,0.05) .. (v);

\draw[edgee] (u) .. controls (1.65,-1.10) and (5.95,-1.10) .. (v);

\node[vertex] at (u) {};
\node[vertex] at (a) {};
\node[vertex] at (b) {};
\node[vertex] at (v) {};

\node[lab,below left=1pt] at (u) {$u$};
\node[lab,above left=1pt] at (a) {$a$};
\node[lab,above right=1pt] at (b) {$b$};
\node[lab,below right=1pt] at (v) {$v$};

\node[lab,teal!75!black] at (3.80,2.05) {$S$};
\node[lab,blue!70!black] at (3.80,0.35) {$T$};
\node[lab,red!75!black] at (3.80,-.55) {$e$};

\node[font=\normalsize] at (3.80,-1.55)
{$R=P_i[u,a]\cup\{e\}\cup P_i[b,v]$, and $S\cup T\cup R$ forms a theta graph.};
\end{tikzpicture}
\end{center}
Let
\[
R:=P_i[u,a] \cup \{e\} \cup P_i[b,v]
\]
be the third $ab$-path. We see that $T$, $S$ and $R$ are pairwise internally vertex-disjoint (as in the figure above), forming a theta subgraph. The three cycles of this theta subgraph are
\[
  C_1=T\cup R=P_i+e,\qquad
  C_2=T\cup S,\qquad
  C_3=S\cup R=P_{i+1}+e.
\]
By the induction hypothesis, $C_1$ is balanced; since $(V,A)$ is balanced and  $C_2$ is its subgraph, $C_2$ is also balanced. It follows from the theta property of the biased graph that $C_3 = P_{i+1} + e$ must be balanced.

\noindent{\bf Step 4: Iteration.}
We start from $P_0 = P$ and iterate the step that transforms $P_i$ into $P_{i+1}$, as illustrated in the following figure:
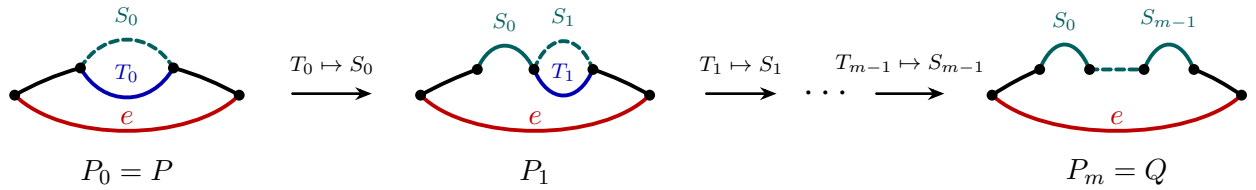
\begin{center}
\begin{tikzpicture}[
  x=1.5cm,y=1cm,scale=1,transform shape,
  line cap=round,line join=round,
  vertex/.style={circle,fill=black,inner sep=1.45pt},
  common/.style={line width=1.45pt,black},
  oldseg/.style={line width=1.45pt,blue!70!black},
  newseg/.style={line width=1.45pt,teal!75!black},
  future/.style={line width=1.45pt,teal!75!black,densely dashed},
  edgee/.style={line width=1.45pt,red!75!black},
  arrow/.style={-{Stealth[length=2.2mm,width=1.7mm]},line width=1pt,black},
  label/.style={font=\small}
]
\begin{scope}[shift={(0,0)}]
\coordinate (p0) at (0,0);
\coordinate (p1) at (0.58,0.36);
\coordinate (p2) at (1.40,0.36);
\coordinate (p3) at (1.98,0);
\draw[common] (p0) .. controls (0.18,0.15) and (0.36,0.27) .. (p1);
\draw[oldseg] (p1) .. controls (0.78,-0.16) and (1.20,-0.16) .. (p2)
  node[midway,above=1pt,font=\scriptsize] {$T_0$};
\draw[future] (p1) .. controls (0.78,0.86) and (1.20,0.86) .. (p2)
  node[midway,above=1pt,font=\scriptsize] {$S_0$};
\draw[common] (p2) .. controls (1.62,0.27) and (1.80,0.15) .. (p3);
\draw[edgee] (p0) .. controls (0.43,-0.63) and (1.55,-0.63) .. (p3);
\foreach \p in {p0,p1,p2,p3} \node[vertex] at (\p) {};
\node[label,red!75!black] at (0.99,-.3) {$e$};
\node[label] at (0.99,-1.02) {$P_0=P$};
\end{scope}

\draw[arrow] (2.45,0.02) -- (3.15,0.02)
  node[midway,above=3pt,font=\scriptsize] {$T_0\mapsto S_0$};

\begin{scope}[shift={(3.58,0)}]
\coordinate (q0) at (0,0);
\coordinate (q1) at (0.50,0.34);
\coordinate (q2) at (1.00,0.34);
\coordinate (q3) at (1.52,0.34);
\coordinate (q4) at (2.02,0);
\draw[common] (q0) .. controls (0.17,0.14) and (0.33,0.25) .. (q1);
\draw[newseg] (q1) .. controls (0.62,0.76) and (0.86,0.76) .. (q2)
  node[midway,above=1pt,font=\scriptsize] {$S_0$};
\draw[oldseg] (q2) .. controls (1.13,-0.12) and (1.39,-0.12) .. (q3)
  node[midway,above=1pt,font=\scriptsize] {$T_1$};
\draw[future] (q2) .. controls (1.13,0.84) and (1.39,0.84) .. (q3)
  node[midway,above=1pt,font=\scriptsize] {$S_1$};
\draw[common] (q3) .. controls (1.70,0.25) and (1.86,0.14) .. (q4);
\draw[edgee] (q0) .. controls (0.42,-0.63) and (1.60,-0.63) .. (q4);
\foreach \p in {q0,q1,q2,q3,q4} \node[vertex] at (\p) {};
\node[label,red!75!black] at (1.01,-.3) {$e$};
\node[label] at (1.01,-1.02) {$P_1$};
\end{scope}

\draw[arrow] (6.10,0.02) -- (6.72,0.02)
node[midway,above=3pt,font=\scriptsize] {$T_1\mapsto S_1$};
\node[font=\large] at (7.18,0.02) {$\cdots$};
\draw[arrow] (7.60,0.02) -- (8.20,0.02)
node[midway,above=3pt,font=\scriptsize] {$T_{m-1}\mapsto S_{m-1}$};

\begin{scope}[shift={(8.62,0)}]
\coordinate (r0) at (0,0);
\coordinate (r1) at (0.42,0.34);
\coordinate (r2) at (0.86,0.34);
\coordinate (r3) at (1.34,0.34);
\coordinate (r4) at (1.78,0.34);
\coordinate (r5) at (2.20,0);
\draw[common] (r0) .. controls (0.14,0.14) and (0.28,0.25) .. (r1);
\draw[newseg] (r1) .. controls (0.54,0.78) and (0.74,0.78) .. (r2)
  node[midway,above=1pt,font=\scriptsize] {$S_0$};
\draw[future] (r2) -- (r3);
\draw[newseg] (r3) .. controls (1.46,0.78) and (1.66,0.78) .. (r4)
  node[midway,above=1pt,font=\scriptsize] {$S_{m-1}$};
\draw[common] (r4) .. controls (1.94,0.25) and (2.06,0.14) .. (r5);
\draw[edgee] (r0) .. controls (0.46,-0.63) and (1.74,-0.63) .. (r5);
\foreach \p in {r0,r1,r2,r3,r4,r5} \node[vertex] at (\p) {};
\node[label,red!75!black] at (1.10,-.3) {$e$};
\node[label] at (1.10,-1.02) {$P_m=Q$};
\end{scope}

\node[label,text width=17.5cm] at (6.25,-1.88)
{The blue subpath is replaced with the teal subpath, which transforms the balanced cycle\\ $P_i+e$ into the balanced cycle $P_{i+1}+e$.};
\end{tikzpicture}
\end{center}

Notice that every switch is nontrivial, and hence each switch strictly extends the common initial subpath between the current path and the fixed path $Q$. Since $Q$ has only finitely many vertices, the process terminates after a finite number of switches, yielding:
\[
  P = P_0,\ P_1,\ \ldots,\ P_m = Q,
\]
with $P_i + e$ balanced for every $i$. In particular, $Q + e = P_m + e$ is balanced. This proves the assertion.

Finally, let $C$ be any cycle in $(V,A\sqcup\{e\})$. If $e \notin C$, then $C$ is balanced because $A$ is balanced. If $e \in C$, then $C - e$ is a $uv$-path in $(V,A)$, and our earlier assertion implies that $C$ is balanced. Therefore, every cycle in $(V, A \sqcup \{e\})$ is balanced, i.e., $A \sqcup \{e\}$ is balanced. 
\end{proof}
With \Cref{Balanced}, we now proceed to show that the triple $(E,\mathcal{C}_{\Omega},r_{\Omega})$ is a semimatroid.
\begin{proposition}\label{PropBiasedSemi}
Let $\Omega=(G,\mathcal{B})$ be a biased graph with the underlying graph $G=(V,E)$. Then the triple $(E,\mathcal{C}_{\Omega},r_{\Omega})$ is a semimatroid.
\end{proposition}
\begin{proof}
Note that the pair $(E,r_G)$ defines the cycle matroid of $G$, and $r_\Omega$ is the restriction of the rank function $r_G$ to the balanced edge sets. Therefore, the function $r_\Omega$ automatically satisfies the axioms (R1)--(R3) by \cite[Theorem 1.3.2]{Oxley2011}.

We now verify (SR4). Let $A,B\in\mathcal{C}_{\Omega}$ satisfy $r_{\Omega}(A)=r_{\Omega}(A\cap B)$. Since $r_G(X)=|V|-c(X)$ for all $X\subseteq E$ and $r_\Omega(X)=r_G(X)$ for all $X\in\mathcal{C}_\Omega$, we immediately obtain $c(A)=c(A\cap B)$. Together with the inclusion relation $A\cap B\subseteq A$, this implies that the spanning subgraphs $(V,A)$ and $(V,A\cap B)$ induce the same connected vertex partition of $V$, where each block is the vertex set of a connected component of $(V,A)$. Consequently, for every edge $e=uv\in A\setminus B$, there exists a $uv$-path $P_e$ in $(V,A\cap B)$. Write $A\setminus B=\{e_1,\ldots,e_k\}$ and set $X_0=B$, $X_i=B\sqcup\{e_1,\ldots,e_i\}$ for $i=1,\ldots k$. We prove by induction that each $X_i$ is balanced. The base case $X_0=B$ is clear. Suppose $X_{i-1}$ is balanced. As $P_{e_i}$ is a path in $(V,A\cap B)$ and $A\cap B\subseteq X_{i-1}$, the path $P_{e_i}$ is contained in $(V,X_{i-1})$, and  $P_{e_i}+e_i$ forms a cycle of $(V,A)$. Since $A$ itself is balanced,  $P_{e_i}+e_i$ is also balanced. It follows from \Cref{Balanced} that $X_i$ is balanced. Thus, $X_k = A\cup B$ is balanced, which completes the proof of (SR4).

For (SR5), suppose $A,B\in\mathcal{C}_{\Omega}$ satisfy $r_{\Omega}(A)<r_{\Omega}(B)$. We claim that some edge $e\in B\setminus A$ joins two different connected components of $(V,A)$. Otherwise, $(V,A)$ and $(V,A\sqcup(B\setminus A)) = (V,A\cup B)$ have the same number of connected components, and hence $r_G(A\cup B)=r_G(A)$. Since $B\subseteq A\cup B$, the monotonicity of $r_G$ gives $r_{\Omega}(B)=r_G(B)\le r_G(A\cup B)=r_G(A)=r_{\Omega}(A)$, which contradicts $r_{\Omega}(A)<r_{\Omega}(B)$. Now choose such an edge $e$. By the choice of $e$, it connects two different connected components of $(V,A)$, so $(V,A\sqcup\{e\})$ has no new cycles compared with $(V,A)$. Consequently, every circle in $A\sqcup\{e\}$ is already a circle of $A$, which is balanced by the assumption that $A\in\mathcal{C}_{\Omega}$. Hence $A\sqcup\{e\}$ is also balanced, i.e., $A\sqcup\{e\}\in\mathcal{C}_{\Omega}$, which completes the proof of (SR5).
\end{proof}

Following \Cref{PropBiasedSemi}, the specialization of \Cref{Result1} to biased graphs is the next result.
\begin{corollary}[Generalized Brylawski Identities for Biased Graphs]\label{CorBiased} 
Let $\Omega=(G,\mathcal{B})$ be a biased graph with the underlying graph $G=(V,E)$, and $T_{\Omega}^b(x,y)=\sum_{i,j\ge0}t_{ij}x^iy^j$. Then, for all integers $k\geq0$,
\[
\sum_{i=0}^{k}\sum_{j=0}^{k-i}\binom{k-i}{j}(-1)^jt_{ij}=\sum_{n=0}^{|E|}(-1)^{|V|-c(E)+n}b_n\binom{k+n-|V|+c(E)}{k},
\]
where $b_n:=\#\{A\in\mathcal{C}_{\Omega}:|A|=n\}$.
\end{corollary}

Signed graphs and gain graphs form important subclasses of biased graphs. As a direct consequence, \Cref{CorBiased} holds for their balanced Tutte polynomials. Finally, if every circle of $G$ is balanced, then $\mathcal{C}_{\Omega}=2^E$, and $T_{\Omega}^b(x,y)$ is the ordinary Tutte polynomial $T_G(x,y)$ of $G$. In this case, $b_n = \binom{|E|}{n}$. Following the same reasoning as in \eqref{A=B}, \Cref{CorBiased} simplifies to the generalized Brylawski identity for graphs:
\[
\sum_{i=0}^{k}\sum_{j=0}^{k-i}\binom{k-i}{j}(-1)^jt_{ij}=(-1)^{|E|-|V|+c(E)}\binom{k-|V|+c(E)}{k-|E|},
\]
which was first established in \cite[Theorem 1.2]{BKCP2023}.
\section*{Acknowledgements}
This work is supported by the Guangdong Basic and Applied Basic Research Foundation (Grant No. 2026A1515012237, Grant No. 2026A1515012543). 

\end{document}